\documentclass[11pt,english]{article}
\usepackage[utf8]{inputenc}
\usepackage[T1]{fontenc}
\usepackage{babel}
 \usepackage{amsmath}
 \usepackage{mathtools}
\usepackage{amsfonts}
\usepackage{amssymb}
\usepackage{graphicx}
\usepackage{amsthm}
\usepackage{bbm}
\usepackage{mathtools}
\usepackage[colorlinks=true,
            linkcolor=blue,
            urlcolor=blue,
            citecolor=blue]{hyperref}
\usepackage{authblk}
\usepackage{geometry}
\usepackage{algorithm}
\usepackage[noend]{algpseudocode}
\usepackage[vcentermath]{youngtab}
\usepackage{pgfplots}
\pgfplotsset{compat=newest}
\usetikzlibrary{calc}
\usepackage{mathrsfs}
 \author{Mohamed Slim Kammoun\footnote{mkammoun@math.univ-toulouse.fr}}
\affil{IMT, Université de Toulouse}

\title{A note on monotone subsequences and the RS image of invariant random permutations with macroscopic number of fixed points}
\newtheorem{theorem}{Theorem}
\newtheorem{corollary}[theorem]{Corollary}
\newtheorem{lemma}[theorem]{Lemma}
\newtheorem{proposition}[theorem]{Proposition}
\newtheorem{conjecture}[theorem]{Conjecture}
\theoremstyle{definition}

\usepackage{tikz}
\usetikzlibrary{automata, positioning}
\usepackage{natbib}
\begin{document}
\maketitle
\begin{abstract} 

The work of \cite{MR0480398,LOGAN1977206} established that the shape of the scaled random young diagram in Russian notation, as determined by the Plancherel measure, converges to a deterministic shape.
In this article, we focus on the scenario where the number of fixed points is substantial. We provide evidence that, subject to specific requirements on the total number of cycles, the limiting shape is a scaled version of the Vershik-Kerov-Logan-Shepp limiting shape. Additionally, we identify certain limiting regimes that resemble those in \cite*{wel}.
Furthermore, we enhance the existing results on Tracy-Widom universality classes for $\beta \in {1,2,4}$ for monotone subsequences.

\end{abstract}
\section{Main results}
\paragraph{}

Let $\mathfrak{S}_n$ be the  group of permutations of $\{1,\dots,n\}$.  We will use in this paper  the well-known application on the symmetric group $\mathfrak{S}_n$ with values in $\mathbb{Y}_n$ known as the shape of the image of a permutation $\sigma$ by the Robinson–Schensted correspondence (RS) \citep*{RSKR,MR0121305} or the Robinson–Schensted–Knuth correspondence \citep*{RSKK}. We denote it by $$\lambda(\sigma)=\{\lambda_i(\sigma)\}_{i\geq 1}.$$ We will not include here algorithmic details. For further reading, we recommend \citep*[Chapter~3]{Sagan2001} and \citep{MR3468738}. A remarkable feature of the shape of RS is that it is entirely determined by monotone subsequences. Indeed,  if we denote by \begin{align*}
\mathfrak{I}_1(\sigma):&=\{s\subset\{1,2,\dots,n\};\; \forall i,j \in s,\; (i-j)(\sigma(i)-\sigma(j))\geq 0 \},
\\ \mathfrak{D}_1(\sigma):&=\{s\subset\{1,2,\dots,n\};\; \forall i,j \in s,\; (i-j)(\sigma(i)-\sigma(j))\leq 0 \},
\\\mathfrak{I}_{k+1}(\sigma):&=\{s\cup s',\; s\in \mathfrak{I}_k,\;s'\in \mathfrak{I}_1\},
\\ \mathfrak{D}_{k+1}(\sigma):&=\{s\cup s',\; s\in \mathfrak{D}_k,\;s'\in \mathfrak{D}_1\},
\end{align*}
we have the following.
\begin{proposition}
    \label{RSKLEMMA} \citep*{GREENE1974254}
For any permutation $ \sigma\in \mathfrak{S}_n$,
\begin{align*}
\max_{s\in \mathfrak{I}_i(\sigma)} |s| =\sum_{k=1}^i \lambda_k(\sigma), \quad
\max_{s\in \mathfrak{D}_i(\sigma)} |s| =\sum_{k=1}^i \lambda'_k(\sigma).
\end{align*}
\end{proposition}

In particular, $\lambda_1(\sigma)$ is the length of the longest increasing subsequence and $ \lambda'_1(\sigma)=\ell(\lambda(\sigma))$  is the length of the longest decreasing subsequence.
\paragraph*{} This result is proved first by \cite{GREENE1974254} (see also \citep*[Theorem 3.7.3]{Sagan2001}). Let now $L_{\lambda(\sigma)}$ be the height function of $\lambda(\sigma)$   rotated by $\frac{3\pi}{4}$  and extended by the function $x\mapsto |x|$ to obtain a function defined on $\mathbb{R}$. For example,  if $\lambda(\sigma)=(7,5,2,1,1,\underline{0})$ the associated function $L_{\lambda(\sigma)}$ is represented by Figure \ref{figL}.
To introduce our main result we need to introduce some additional notations. Given $\sigma \in \mathfrak{S}_n$, we denote by  $\#(\sigma)$ the number of cycles of $\sigma$ 
and by $\#_1(\sigma)$ its  total number of fixed points.  For example, for the permutation \begin{equation*}\sigma=\begin{pmatrix}
 1& 2 & 3 & 4 & 5 & 6 \\ 
 5& 3 & 2 & 1 & 4 & 6
\end{pmatrix},\end{equation*}
 $\#(\sigma)=3$ and $\#_1(\sigma)=1$.
In our work,
we are interested in conjugacy invariant  permutations $\sigma_n$  i.e. for any permutation $\rho \in \mathfrak{S}_n$,
$\rho\sigma_n\rho^{-1}$ has the same distribution as $\sigma_n$. Our main result is the following.


\begin{figure}[]
\centering
\begin{tikzpicture}    [/pgfplots/y=0.4cm, /pgfplots/x=0.4cm]
      \begin{axis}[
    axis x line=center,
    axis y line=center,
    xmin=0, xmax=10,
    ymin=0, ymax=10, clip=false,
    ytick={0},
	xtick={0},
    minor xtick={0,1,2,3,3,4,5,6,7,8,9},
    minor ytick={0,1,2,3,3,4,5,6,7,8,9},
    grid=both,
    legend pos=north west,
    ymajorgrids=false,
    xmajorgrids=false, anchor=origin,
    grid style=dashed    , rotate around={45:(rel axis cs:0,0)}
,
]

\addplot[
    color=blue,
        line width=3pt,
    ]
    coordinates {
    (0,10)(0,7)(1,7)(1,5)(2,5)(2,2)(3,2)(3,1)(5,1)(5,0)(10,0)
    };
 
\end{axis}
\begin{axis}[
	axis x line=center,
    axis y line=center,
    xmin=-7.07, xmax=7.07,
    ymin=0, ymax=8, anchor=origin, clip=false,
    xtick={-7,-6,-5,-4,-3,-2,-1,0,1,2,3,4,5,6,7},
    ytick={0,1,2,3,3,4,5,6,7,8},
    legend pos=north west,
    ymajorgrids=false,
    xmajorgrids=false,rotate around={0:(rel axis cs:0,0)},
    grid style=dashed];
\end{axis}
    \end{tikzpicture}
    \caption{ $L_{(7,5,2,1,1,\underline{0})}$}
     \label{figL}
\end{figure}
\begin{theorem} \label{thm2}
Assume that for any $n$, $\sigma_n$ is a conjugacy invariant random permutation of size $n$ and
\begin{equation}
\frac{\min\left({n-\#_1(\sigma_n^2)}\,;\,{\#(\sigma_n)- \#_1(\sigma_n)} \right)}{n}  \xrightarrow[n\to\infty]{\mathbb{P}} 0.
\end{equation}
Then
\begin{align}
\sup_{s\in \mathbb{R}} \left|\frac{1}{2\sqrt{n}}L_{\lambda(\sigma_n)}\left({2s\sqrt{n}}\right)-\sqrt{1-\frac{\#_1(\sigma_n)}{n}}\Omega\left(s \sqrt{1-\frac{\#_1(\sigma_n)}{n}} \right)\right|\xrightarrow[n\to\infty]{\mathbb{P}}0.
\end{align}
where  
\begin{align*}
\Omega(s):=\begin{cases}
\frac{2}{\pi}(s\arcsin({s})+\sqrt{1-s^2}) & \text{ if } |s|<1 \\ 
|s| & \text{ if } |s|\geq 1 
\end{cases}.
\end{align*}
\end{theorem}
Our main result can be seen as a generalization various previous works.  
The typical shape  when the permutation is uniform (resp. a uniform involution)  was studied separately by \cite{LOGAN1977206} and \cite{MR0480398} (resp \cite{Mliot2011KerovsCL}). In a previous paper \citep{2018arXiv180505253S}, we studied the case where the number of cycles is microscopic. 
  For the longest monotone subsequences, we have the following.
\begin{theorem} \label{thm3}
Assume that for any $n$, $\sigma_n$ is a conjugacy invariant random permutation of size $n$  and that $n-\#_1(\sigma)\xrightarrow[n\to\infty]{\mathbb{P}} \infty$

\begin{itemize}
\item 
if 
\begin{equation}\label{cond3.1}
\frac{\#(\sigma_n)- \#_1(\sigma_n)}{{n^{\frac16}}}  \xrightarrow[n\to\infty]{\mathbb{P}} 0.
\end{equation}
Then 
\begin{align} \label{EQ3.1}
\frac{\ell(\lambda(\sigma_n))-2\sqrt{n-\#_1(\sigma_n)}}{(n-\#_1(\sigma_n))^\frac{1}{6}} \xrightarrow[n\to\infty]{d}TW_2.
\end{align}
\item if 
\begin{equation}\label{cond3.2}
\frac{n- \#_1(\sigma_n^2)}{{n^{\frac16}}}  \xrightarrow[n\to\infty]{\mathbb{P}} 0.
\end{equation}
Then 
\begin{equation}
\label{EQ3.2} \frac{\ell(\lambda(\sigma_n))-2\sqrt{n}}{n^\frac{1}{6}} \xrightarrow[n\to\infty]{d}TW_1.
\end{equation}
\item if 
\begin{equation}\label{cond3.3}
\frac{n- 2\#_2(\sigma_n)}{{n^{\frac16}}}  \xrightarrow[n\to\infty]{\mathbb{P}} 0.
\end{equation}
Then 
\begin{equation} \label{EQ3.3}
    \frac{\lambda_1(\sigma_n)-2\sqrt{n}}{n^\frac{1}{6}} \xrightarrow[n\to\infty]{d}TW_4.
\end{equation}
    \item if 
\begin{equation} \label{cond3.4}
\frac{\min\left({n-\#_1(\sigma_n^2)}\,;\,{\#(\sigma_n)- \#_1(\sigma_n)} \right)}{\sqrt{n}}  \xrightarrow[n\to\infty]{\mathbb{P}} 0.
\end{equation}
Then 
\begin{align} \label{EQ3.4}
\frac{\ell(\lambda(\sigma_n))}{\sqrt{n-\#_1(\sigma_n)}} \xrightarrow[n\to\infty]{\mathbb{P}}2    
\end{align}

\end{itemize}
\end{theorem}
When  $\#_1(\sigma_n)= (\theta n/\log(n)) (1+o(1))$, we recover a regime that is similar to that obtained by \cite*{wel} by a direct application of Theorem~\ref{thm2} and Theorem~\ref{thm3}.
\begin{corollary}\label{corW}
Suppose that for any $n$, $\sigma_n$ is a conjugacy invariant random permutation of size $n$. In addition,   assume that $\frac{\#_1(\sigma)\log(n)}{\theta n}\xrightarrow[n\to\infty]{\mathbb{P}} 1$  and that 
$$\frac{\#(\sigma_n)- \#_1(\sigma_n)}{{n^{\frac12}}}  \xrightarrow[n\to\infty]{\mathbb{P}} 0.$$

We have then,
\begin{itemize}
\item
$$\frac{\ell(\lambda(\sigma_n))}{\sqrt{n}} \xrightarrow[n\to\infty]{\mathbb{P}}2,$$

\item 
$$\frac{\lambda_1(\sigma_n)\log(n)}{\theta n} \xrightarrow[n\to\infty]{\mathbb{P}}1,$$
\item 
$$\sup_{s\in \mathbb{R}} \left|\frac{1}{\sqrt{2n}}L_{\lambda(\sigma_n)}\left({s}{\sqrt{2n}}\right)-\Omega(s)\right|\xrightarrow[n\to\infty]{\mathbb{P}}0.$$
\end{itemize}
\end{corollary}
 We conjecture that $\lambda_2$ have the same behaviours as in \citep{wel}  i.e. 
$$\frac{\lambda_2(\sigma_n)}{\sqrt{n}} \xrightarrow[n\to\infty]{\mathbb{P}}2,$$
but we can prove only the following.
\begin{lemma} \label{main_lemma}
Under the same hypothesis than Corollary \ref{corW},
for any $\varepsilon>0$
$$\mathbb{P} \left(2-\varepsilon<\frac{ \lambda_2(\sigma_n)}{\sqrt{n}} <4+\varepsilon \right)\xrightarrow[n\to\infty]{}1$$
\end{lemma}

Recently, \cite{borga2023powerlaw} and \cite{dubach2023locally} gave a non trivial  regimes where the longest increasing subsequence  $\lambda_1(\sigma_n)$ behave like $n^{\beta}$.
In our case similar regimes appear  for the same reasons as before  when $$\frac{\#_1(\sigma)}{n^\beta}\xrightarrow[n\to\infty]{\mathbb{P}} c$$ and 
$$\frac{\#(\sigma)-\#_1 (\sigma)}{n^\beta}\xrightarrow[n\to\infty]{\mathbb{P}} 0.$$

For general conjugacy invariant permutations, we do conjecture that a university phenomenon  holds for any conjugacy invariant permutations.  In particular, we conjecture that

\begin{conjecture}\label{conj}
Assume that for any $n$, $\sigma_n$ is a conjugacy invariant random permutation of size $n$.
Then
\begin{align}
\sup_{s\in \mathbb{R}} \left|\frac{1}{2\sqrt{n}}L_{\lambda(\sigma_n)}\left({2s\sqrt{n}}\right)-\sqrt{1-\frac{\#_1(\sigma_n)}{n}}\Omega\left(s \sqrt{1-\frac{\#_1(\sigma_n)}{n}} \right)\right|\xrightarrow[n\to\infty]{\mathbb{P}}0.
\end{align}
\end{conjecture}
Conjecture \ref{conj} is equivalent to the following conjecture. 

\begin{conjecture} \label{conj1}
Assume that for any $n$, $\sigma_n$ is a conjugacy invariant random permutation of size $n$. Suppose that 
\begin{align}\label{H1}
    \#_1(\sigma_n)\overset{a.s}=0.
\end{align}
Then
\begin{align} \label{VKLGS}
\sup_{s\in \mathbb{R}} \left|\frac{1}{2\sqrt{n}}L_{\lambda(\sigma_n)}\left({2s\sqrt{n}}\right)-\Omega\left(s\right)\right|\xrightarrow[n\to\infty]{\mathbb{P}}0.
\end{align}
\end{conjecture}

\textbf{Acknowledgements:} The author want to thank 
Guillaume Chapuy, Victor Dubach, Valentin Féray,   Baptiste Louf, Pierre-Loïc Méliot and Harriet Walsh
for very useful discussions.   This work is supported by Labex CIMI (ANR-11-LABX-0040).

\section{Combinatorial control}


The objective of this section is to establish two combinatorial lemmas, namely Lemma \ref{34} and Lemma \ref{l15}, which provide control over the shape of RS.

\begin{lemma} \label{34}
Let $n,m\in \mathbb{N}^*$, $\lambda=(\lambda_i)_{i\geq1}\in\mathbb{Y}_n$, $\mu=(\mu_i)_{i\geq1}\in\mathbb{Y}_m$. Then,  
\begin{align}\label{disVC}
\sup_{s\in\mathbb{R}}\left|L_\lambda(s)-L_\mu(s)\right| \leq \min_{l\geq0} 2\sqrt{ \max_{m\geq l+1} \left|\sum_{k=l+1}^m( \lambda_k-\mu_k)\right|}+\sqrt{2}l.
\end{align}
\end{lemma}
Lemma \ref{34} can be seen as a generalization of 
\citep[Lemma~34]{2018arXiv180505253S}.
Before giving its proof  we introduce some notations. 
Let $(O,\vec{x},\vec{y})$ be  the canonical frame of the Euclidean plane and $\vec{u}:=\frac{\sqrt{2}}{2}(\vec{x}+\vec{y})$, $\vec{v}:=\frac{\sqrt{2}}{2}(\vec{y}-\vec{x})$. Let $\lambda \in \mathbb{Y}_n$. Using the convention $\lambda_0=\infty$, let $\mathscr{C}_\lambda$ be the curve obtained by connecting  the points with coordinates   $(0,\lambda_0),(0,\lambda_1), (1,\lambda_1),(1,\lambda_2),\dots,$ $ (i,\lambda_{i}),(i,\lambda_{i+1}),\dots$ in the axes system $(O,\overrightarrow{u},\overrightarrow{v})$ as in Figure \ref{figL31}.
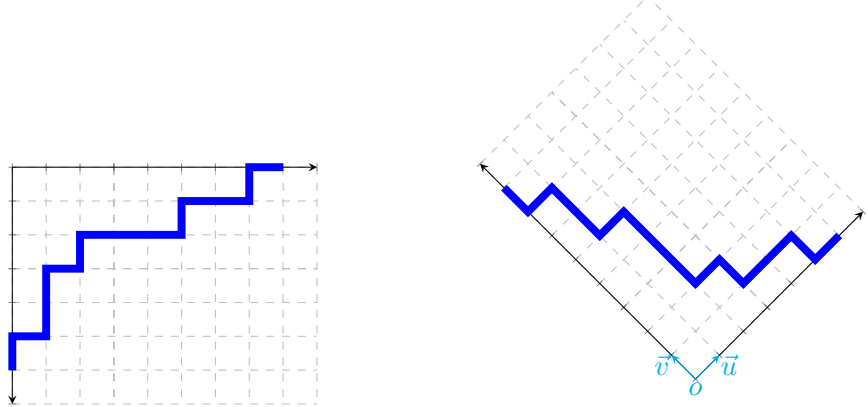
\begin{figure}[H]
\centering
\begin{tikzpicture}[/pgfplots/y=0.45cm, /pgfplots/x=0.45cm]
      \begin{axis}[
    axis x line=center,
    axis y line=center,
    xmin=0, xmax=7,
    ymin=0, ymax=9, clip=false,
    ytick={0},
	xtick={0},
    minor xtick={0,1,2,3,3,4,5,6,7,8,9},
    minor ytick={0,1,2,3,3,4,5,6,7,8,9},
    grid=both,
    legend pos=north west,
    ymajorgrids=false,
    xmajorgrids=false, anchor=origin,
    grid style=dashed    , rotate around={-90:(rel axis cs:0,0)},
]

\addplot[
    color=blue,
        line width=3pt,
    ]
    coordinates {
    (0,8)(0,7)(1,7)(1,5)(2,5)(2,2)(3,2)(3,1)(5,1)(5,0)(6,0)
    };
\end{axis}
    \end{tikzpicture}\;  \;  \; \; \;  \; \; \;  
    \begin{tikzpicture}[/pgfplots/y=0.45cm, /pgfplots/x=0.45cm]

 \begin{axis}[
    axis x line=center,
    axis y line=center,
    xmin=0, xmax=7,
    ymin=0, ymax=9, clip=false,
    ytick={0},
	xtick={0},
    minor xtick={0,1,2,3,3,4,5,6,7,8,9},
    minor ytick={0,1,2,3,3,4,5,6,7,8,9},
    grid=both,
    legend pos=north west,
    ymajorgrids=false,
    xmajorgrids=false, 
    grid style=dashed    , rotate around={45:(rel axis cs:0,0)},
]

\addplot[
    color=blue,
        line width=3pt,
    ]
    coordinates {
    (0,8)(0,7)(1,7)(1,5)(2,5)(2,2)(3,2)(3,1)(5,1)(5,0)(6,0)
    };
   \addplot[cyan,
        quiver={u=\thisrow{u},v=\thisrow{v}},
        -stealth]
    table
    {
    x y u v
    0 0 1 0
    0 0 0 1
    };
        \node[cyan] at (axis cs: -0.4,1) {$\vec{v}$};
 \node[cyan] at (axis cs: 1,-0.4) {$\vec{u}$};
                 \node[cyan] at (axis cs: -0.2,-0.2) {$o$};
\end{axis}
    \end{tikzpicture}
    \caption{  $\mathscr{C}_\lambda$ for $\lambda=(7,5,2,1,1,\underline{0})$}
     \label{figL31}
\end{figure}
 By construction $\mathscr{C}_\lambda$ is the curve of $L_\lambda$. This yields the following two lemmas.
\begin{lemma} \citep[Lemma~33]{2018arXiv180505253S} \label{lemmainq}
Let  $\alpha,\beta \in \mathbb{N}$ and $A$ the point  such that  $\overrightarrow{OA}=\alpha\vec{u}+\beta\vec{v}$. If $A\in \mathscr{C}_\lambda$, then
\begin{equation}\label{la1}
\lambda_{\alpha+1}\leq \beta\leq \lambda_\alpha.
\end{equation}
Moreover, for all $i\in\mathbb{Z}$,
\begin{align}\label{eq15}
\frac{\sqrt{2}}{2} L_\lambda\left(\frac{\sqrt{2}}{2}i\right)\pm\frac{i}{2}\in \mathbb{N}.
\end{align}
\end{lemma}

\begin{figure}[H]
\centering
\begin{tikzpicture}    [/pgfplots/y=0.6cm, /pgfplots/x=0.6cm]
      \begin{axis}[
    axis x line=center,
    axis y line=center,
    xmin=0, xmax=10,
    ymin=0, ymax=10, clip=false,
    ytick={0},
	xtick={0},
    minor xtick={0,1,2,3,3,4,5,6,7,8,9},
    minor ytick={0,1,2,3,3,4,5,6,7,8,9},
    grid=both,
    legend pos=north west,
    ymajorgrids=false,
    xmajorgrids=false, anchor=origin,
    grid style=dashed    , rotate around={45:(rel axis cs:0,0)},
]
\addplot[
    color=blue,
        line width=3pt,
    ]
    coordinates {
    (0,9)(0,7)(1,7)(1,5)(2,5)(2,2)(3,2)(3,1)(5,1)(5,0)(9,0)
    };
    \addplot[
    color=green,
        line width=2pt,
    ]
    coordinates {
    (0,9)(0,4)(1,4)(1,4)(2,4)(2,3)(5,3)(5,1)(9,1)(9,0)(9,0)
    };
    \addplot[cyan,
        quiver={u=\thisrow{u},v=\thisrow{v}},
        -stealth]
    table
    {
    x y u v
    0 0 1 0
    0 0 0 1
    };
        \node[cyan] at (axis cs: 0.2,0.8) {$\vec{v}$};
 \node[cyan] at (axis cs: 0.7,0.2) {$\vec{u}$};
                 \node[] at (axis cs: -0.2,-0.2) {$o$};

       \addplot [ mark=*, color=magenta] table {
0 4
1 5
};
    \node[color=magenta] at (axis cs: 0.7,5) {$k_{-4}=-1$};

       \addplot [ mark=*, color=magenta] table {
3 1
5 3
};
    \node[color=magenta] at (axis cs:4,1.8) {$k_2=2$};         
\end{axis}
\begin{axis}[
	axis x line=center,
    axis y line=center,
    xmin=-6, xmax=6,
    ymin=0, ymax=8, anchor=origin, clip=false,
    xtick={-7,-6,-5,-4,-3,-2,-1,0,1,2,3,4,5,6,7},
    ytick={0,1,2,3,3,4,5,6,7,8},
    legend pos=north west,
    ymajorgrids=false,
    xmajorgrids=false,rotate around={0:(rel axis cs:0,0)},
    grid style=dashed,
    ];
       
       \addplot[red,
        quiver={u=\thisrow{u},v=\thisrow{v}},
        -stealth]
    table
    {
    x y u v
    0 0 1 0
    0 0 0 1
    };
        \node[red] at (axis cs: 0.2,1) {$\vec{y}$};
 \node[red] at (axis cs: 0.8,0.25) {$\vec{x}$};

\end{axis}
    \end{tikzpicture}
     \caption{  An example  where $\lambda= {(7,5,2,1,1,\underline{0})}$ and $\mu={{(4,4,3,3,3,1,1,1,1,\underline{0})}}$}
     \label{figL2}
\end{figure}
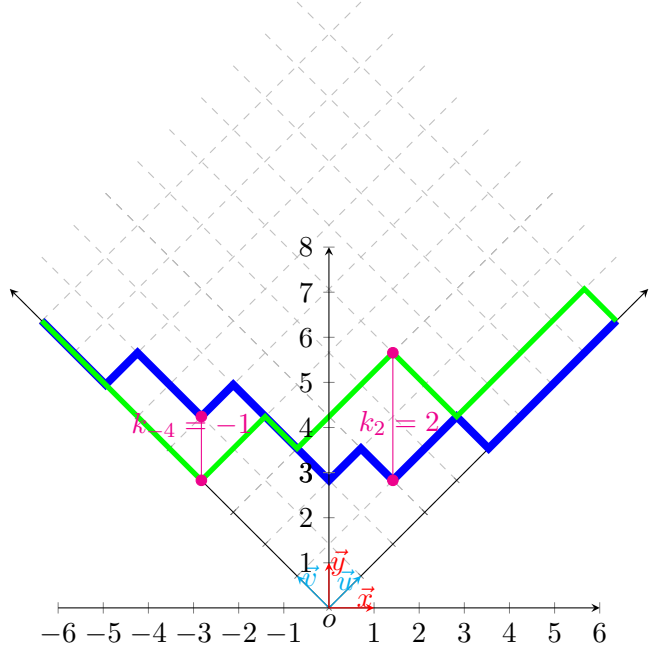
\begin{proof}[Proof of Lamma \ref{34}]
Note that for any $i\in \mathbb{Z}$, $s\mapsto L_{\lambda}(s)$  and  $s\mapsto L_{\mu}(s)$ are affine functions on  $\left[\frac{\sqrt{2}}{2}i,\frac{\sqrt{2}}{2}(i+1)\right]$ 
and thus \eqref{disVC} is equivalent to 
\begin{equation*}
\max_{i\in\mathbb{Z}}\left|L_\lambda\left(\frac{\sqrt{2}}{2}i\right)-L_\mu\left(\frac{\sqrt{2}}{2}i\right)\right|\leq  \min_{l\geq0} 2\sqrt{ \max_{m\geq l+1} \left|\sum_{k=l+1}^m( \lambda_k-\mu_k)\right|}+\sqrt{2}l.
\end{equation*}
Let $i \in \mathbb{Z}$. It follows from \eqref{eq15} that there exists $k_i \in \mathbb{Z}$ such that,  
\begin{align*}
L_{\mu}\left(\frac{\sqrt{2}}{2}i\right)- L_{\lambda}\left(\frac{\sqrt{2}}{2}i\right) = k_i\sqrt{2}.
\end{align*}
By symmetry  we can assume that $k_i>0$. We denote by $$j:={\sqrt{2}L_{\lambda}\left(\frac{\sqrt{2}}{2}i\right)}.$$
Let $A$ and $B$ be the points such that 
\begin{align*}
\overrightarrow{OA}:=\frac{\sqrt{2}}{2}(i\vec{x}+j\vec{y})=\frac{i+j}{2}\vec{u}+\frac{j-i}{2}\vec{v}, \qquad \overrightarrow{OB}=\frac{\sqrt{2}}{2}(i\vec{x}+(j+2k_i)\vec{y})=\frac{i+j+2k_i}{2}\vec{u}+\frac{j-i+2k_i}{2}\vec{v}.
\end{align*}
 Clearly $A\in \mathscr{C}_\lambda$ and $B\in \mathscr{C}_\mu$.    By \eqref{eq15}, $\frac{i+j}{2},\frac{j-i}{2} \in \mathbb{N}$. We can then apply \eqref{la1}.  In the case where $k_i>0$, we have
\begin{align*}
\lambda_{\frac{i+j}{2}+1} \leq \frac{j-i}{2}, \quad
\mu_{\frac{i+j}{2}+k_i} \geq \frac{j-i}{2}+k_i.
\end{align*}
Let $ 0\leq l< k_i$, using the fact that $(\lambda_{l})_{l\geq 1}$ and  of $(\mu_{l})_{l\geq 1}$ are decreasing, we have, 
\begin{align*}
2 \max_{m\geq l+1} \left|\sum_{k=l+1}^m( \lambda_k-\mu_k)\right| &\geq \sum_{m={\max(\frac{i+j}{2}+1},l+1)}^{\frac{i+j}{2}+k_i}\mu_m-\lambda_m \\&\geq \sum_{m={\max(\frac{i+j}{2}+1},l+1)}^{\frac{i+j}{2}+k_i}\mu_{\frac{i+j}{2}+k_i}-\lambda_{\frac{i+j}{2}+1}\geq k_i(k_i-l)\geq (k_i-l)^2. 
\end{align*}

This yields 
\begin{align*}
2\sqrt{ \max_{m\geq l+1} \left|\sum_{k=l+1}^m( \lambda_k-\mu_k)\right|}+\sqrt{2}l \geq \max_{i\in\mathbb{Z}}  \left(\sqrt{2}k_i\right)=\sup_{s\in\mathbb{R}}\left|L_\lambda(s)-L_\mu(s)\right|.
\end{align*}
\end{proof}

A crucial observation for our proof is that the length of the blocks in the RS image of a permutation with a significant number of fixed points can be bounded by the length of the word obtained after removing those fixed points. Let $\sigma$ denote a permutation, $fix(\sigma)$ represent its set of fixed points, and $\tau$ be the  permutation obtained from $\sigma$ after removing $fix(\sigma)$. 
Throughout the rest of this paper, we will use the notation $\sigma=(fix(\sigma),\tau)$ for brevity.
\begin{lemma} \label{l15}
Let $\sigma=(fix(\sigma),\tau)$
\begin{itemize}
\item  
\begin{align}\label{leq1}
    \#_1(\sigma) \leq \lambda_1(\sigma) \leq  \#_1(\sigma)+\lambda_1(\tau)
\end{align}
\item For all $i\geq 2$,
\begin{align}\label{leq2}
    \#_1(\sigma) +\sum_{j=1}^{i-1}\lambda_j(\tau) \leq \sum_{j=1}^i \lambda_j(\sigma) \leq  \#_1(\sigma)+\sum_{j=1}^i\lambda_j(\tau)
\end{align}
\item 
\begin{align}\label{leq3}
    \max_{m\geq 2} \left|\sum_{k=2}^m( \lambda_k(\sigma)-\lambda_k(\tau))\right| \leq \lambda_1(\tau)
\end{align} 
\item \begin{align} \label{leq4}
    \ell(\lambda(\tau)) \le \ell(\lambda(\sigma))\le \ell(\lambda(\tau)) +1
\end{align} 
\end{itemize}

\end{lemma}

\begin{proof} Those  inequalities are an application of Greene's theorem.
\begin{itemize}
    \item We have the inclusion  $fix(\sigma) \subset \mathfrak{I}_1(\sigma)$ which implies the first inequality of~\eqref{leq1}. 
    \item Let  $i_1<i_2<\dots<i_{\lambda_1(\sigma)}$ is an increasing subsequence of $\sigma$ of maximal length. We have
    \begin{align*}\{i_1,i_2,\dots i_{\lambda_1(\sigma)}\}=(\{i_1,i_2,\dots i_{\lambda_1(\sigma)}\}\cap fix(\sigma)) &\cup (\{i_1,i_2,\dots i_{\lambda_1(\sigma)}\}\cap (\{1,2,\dots,n\}\setminus fix(\sigma))),
    \\ (\{i_1,i_2,\dots i_{\lambda_1(\sigma)}\}\cap fix(\sigma)) &\subset fix(\sigma),
    \\ (\{i_1,i_2,\dots i_{\lambda_1(\sigma)}\}\cap (\{1,2,\dots,n\}\setminus fix(\sigma))) &\subset \mathfrak{I}_1(\tau).
    \end{align*}
    This yields the second inequality of \eqref{leq1}
    \item Assume $I \subset \mathfrak{I}_{i-1}(\tau)\subset\mathfrak{I}_{i-1}(\sigma) $. Then  $(I\cup fix(\sigma))\subset \mathfrak{I}_{i}(\sigma)$. This yields the first inequality of \eqref{leq2}.
    \item Let $I \subset  \mathfrak{I}_{i}(\sigma)$ of maximal length. Then
        \begin{align*}I=(I\cap fix(\sigma)) &\cup (I\cap (\{1,2,\dots,n\}\setminus fix(\sigma))),
    \\ (I\cap fix(\sigma)) &\subset fix(\sigma),
    \\ (I\cap (\{1,2,\dots,n\}\setminus fix(\sigma))) &\subset \mathfrak{I}_i(\tau).
    \end{align*}
    This yields the second inequality of \eqref{leq2}.
    \item By subtracting \eqref{leq1} from \eqref{leq2}. We obtain 
    \begin{align*}
    -\lambda_1(\tau)     \leq -\lambda_i(\tau)    \leq \sum_{j=2}^i( \lambda_j(\sigma)-\lambda_j(\tau)) \leq \lambda_1(\tau).
    \end{align*}
    This yields  \eqref{leq3}.
\end{itemize}
\end{proof}

\section{Probabilistic lemmas}
The key probabilistic observations is the following. 
\begin{lemma} \label{lem11}
    If $\sigma_n=(fix(\sigma_n),\tau_n)$ is conjugacy invariant then conditionally on $fix(\sigma_n)$,  $\tau_n$ is conjugacy invariant  i.e.
    for any $S\subset  [n]$, and $\rho$ a permutation of $[n]\setminus S$, and $\pi$ a permutation of size $n$. If $\mathbb{P}(\#_1(\sigma_n)=card(S)) \neq 0$ then 
    $$\mathbb{P}(\sigma_n= \tau | fix(\sigma_n)=S) = \mathbb{P}(\rho\sigma_n\rho^{-1}= \tau | fix(\sigma_n)=S) $$
\end{lemma}
The proofs is a direct result of the definition of conditinal probability and conjugacy invariance. 
This lemma allows us to use the known results for permutations with no (or few) fixed points. We recall those results in Appendix A.
\begin{lemma}\label{l17} Assume that for any $n$, 
$ \sigma_n=(fix(\sigma_n), \tau_n)$ is a conjugacy invariant  involution of size $n$.

 Then for all $\varepsilon>0$,
\begin{align*}
\sup_{s\in \mathbb{R}} \left|\frac{1}{2\sqrt{n}}L_{\lambda(\sigma_n)}\left({2s\sqrt{n}}\right)-\sqrt{1-\frac{\#_1(\sigma_n)}{n}}\Omega\left(s \sqrt{1-\frac{\#_1(\sigma_n)}{n}} \right)\right|\xrightarrow[n\to\infty]{\mathbb{P}}0.
\end{align*}
\end{lemma}
\begin{proof}
Using Lemma \ref{34} we obtain 
\begin{align*}
\sup_{s\in\mathbb{R}}\left|L_{\lambda(\sigma_n)}(s)-L_{\lambda(\tau_n)}\right| \leq \ 2\sqrt{ \max_{m\geq 2} \left|\sum_{k=l+1}^m( \lambda_k(\sigma_n)-\lambda_k(\tau_n))\right|}+\sqrt{2}.
\end{align*}
By Lemma \ref{l15}, we obtain

\begin{align*}
\sup_{s\in\mathbb{R}}\left|L_{\lambda(\sigma_n)}(s)-L_{\lambda(\tau_n)}\right| \leq \ 2\sqrt{\lambda_1(\tau_n)}+\sqrt{2}.
\end{align*}
Observe that conditionally on $fix(\sigma_n)$, $\tau_n$ is a unifom fixed point free involution of size $n-fix(\sigma_n)$.

Consequently, a direct application of Corollary~\ref{Meliot_kam} for any $\varepsilon,\varepsilon'>0$ there exits $n_0$ such that for any $S$ of cardinal at least $n_0$, conditionally on $fix(\sigma_n)=S, $

$$\mathbb{P} \left(\sup_{s\in \mathbb{R}} \left|\frac{1}{2\sqrt{n}}L_{\lambda(\tau_n)}\left({2s\sqrt{n}}\right)-\sqrt{1-\frac{\#_1(\sigma)}{n}}\Omega\left(s \sqrt{1-\frac{\#_1(\sigma_n)}{n}} \right)\right|>\varepsilon\right) <\varepsilon'. $$
This yields Lemma~\ref{l17}.

\end{proof}

\section{Proof of main results}
\begin{proof}[Proof of Theorem~\ref{thm2}]
The idea is to prove the result separately  when 
$\frac{{\#(\sigma_n)- \#_1(\sigma_n)} }{n}  \xrightarrow[n\to\infty]{\mathbb{P}} 0$ and when $\frac{{n-\#_1(\sigma_n^2)} }{n}  \xrightarrow[n\to\infty]{\mathbb{P}} 0.$

\begin{itemize}
    \item Assume that 
    $\frac{{\#(\sigma_n)- \#_1(\sigma_n)} }{n}  \xrightarrow[n\to\infty]{\mathbb{P}} 0$ 
    The proof in this case is similar to that of Lemma~\ref{l17}. Indeed, 
    \begin{align*}
\sup_{s\in\mathbb{R}}\left|L_{\lambda(\sigma_n)}(s)-L_{\lambda(\tau_n)}\right| \leq \ 2\sqrt{\lambda_1(\tau_n)}+\sqrt{2}.
\end{align*}
holds true. 
Observe that conditionally on $fix(\sigma_n)$, $\tau_n$ is a uniform cyclic permutations.

Moreover, Lemma~\ref{VCthm}  imply that for a uniform cyclic permutation $\pi_n$,
\begin{align*} 
\sup_{s\in \mathbb{R}} \left|\frac{1}{2\sqrt{n}}L_{\lambda(\pi_n)}\left({2s\sqrt{n}}\right)-\Omega\left(s\right)\right|\xrightarrow[n\to\infty]{\mathbb{P}}0.
\end{align*}

The conclusion is the same as Lemma~\ref{l17}

\item Assume that, $\frac{{n-\#_1(\sigma_n^2)} }{n}  \xrightarrow[n\to\infty]{\mathbb{P}} 0.$. 

By conditioning on  on $\#_2(\sigma_n)=k$
This is a direct application of  \citep[Theorem 30]{KAMMOUN202276} (see the first example of section 4.2) by choosing   \begin{itemize}
    \item $(V_n,E'_n)={\mathcal{G}_\mathfrak{S}}_n$,  The cayley graph of the symmetric groupe. 
\item $I_n$=$\mathbb{Y}_n$, The set of Young diagrams 
 \item $V_n^i=\{\sigma ;$ Is the set of permutations with a cyclic structure $i$. 
 \item $Class(\sigma)$, is the cycle structure of $\sigma$
 \item $i^*_n=(\underbrace{2,\dots,2}_{k \text{ times}}, \underbrace{1,\dots,1}_{n-2k \text{ times}})$  the Young diagram k  rows of length $2$, 
 \item In this case it is easy to see that  $\underline{d}(\sigma)= \sum_{j>2} j\#j(\sigma)={n-\#_1(\sigma^2)}$.
\end{itemize}
The conditions (17)-(20) of \citep[Theorem 30]{KAMMOUN202276} have already been proved in \citep{KAMMOUN202276}.  
The convergence of reference measure is proved is Lemma~\ref{l17}.
The convergence (22) is trivial since the LHS is equal to $0$ almost surely and the convergence (23) is a direct application of \citep[Lemma 3.7]{2018arXiv180505253S}.   
\end{itemize}
This concludes the proof.
\end{proof}
\begin{proof}[Proof of Theorem~\ref{thm3}] We will each equation seperatly.
  \begin{itemize} 
\item  By \eqref{leq4},  we obtain 
 \begin{align*} 
    \ell(\lambda(\tau_n)) \le \ell(\lambda(\sigma_n))\le \ell(\lambda(\tau_n)) +1
\end{align*} 
By conditioning on $fix(\sigma_n)=S$, 
      and under \eqref{cond3.1},  by Lemma~\ref{lem11},
      $\tau_n$ satisfies the conditions of Lemma~\ref{l16} and 
      then 
for all $\varepsilon>0$, there exists some $n_0$ such that if $card(S)>n_0$, 
\begin{align*}
\mathbb{P}\left|\left(\frac{\ell(\lambda(\tau_n))-2\sqrt{card(S)}}{card(S)^\frac 16}  \leq  s\right)-F_2(s) \right|<\varepsilon.
\end{align*}
Which implies  \eqref{EQ3.1}. 

\item 
By conditioning on  on $\#_2(\sigma_n)=k$.
Note that the convergence in probability in  \citep[Theorem 30]{KAMMOUN202276}  can be replaced by a convergence in distribution and by applying it to $\ell(\lambda(\tau_n))$ by choosing   $V_n$,
 $I_n$, $V_n^i $, $Class(\sigma)$ and 
 $i^*_n$ as in the proof of Theorem~\ref{thm2}.
The convergence of reference measure is given by Proposition~\ref{BRProp}.
The convergence (22) is trivial since the LHS is equal to $0$ almost surely and under \eqref{cond3.2}, the convergence (23) is a direct application of \citep[Lemma 2.5]{KAMMOUN202276}.   This implies \eqref{EQ3.2}.
\item
We apply \citep[Theorem 30]{KAMMOUN202276}
to $\lambda_1(\tau_n)$  by choosing   $V_n$,
 $I_n$, $V_n^i $, $Class(\sigma)$ as in the proof of Theorem~\ref{thm2} and  by fixing
 $i^*_{2p}$ the partition with all blocks equal to $2$ and  $i^*_{2p+1}$ the partition with $p$ blocks equal to $2$ and 1 block equal to one. 
In this case it is easy to see that  $$\underline{d}(\sigma) \leq 1+ \sum_{j \neq 2} j\#j(\sigma)= 1+ {n-2\#_2(\sigma)}.$$ 
The convergence of reference measure is given by Proposition~\ref{BRProp}. 
Under \eqref{cond3.3}, the convergence (23) is a direct application of \citep[Lemma 2.5]{KAMMOUN202276}.   This implies \eqref{EQ3.3}.
      \item 
      
      Similarly to the proof of Theorem~\ref{thm2}, the idea is to prove the result separately  when 
$$\frac{{\#(\sigma_n)- \#_1(\sigma_n)} }{\sqrt{n}}  \xrightarrow[n\to\infty]{\mathbb{P}} 0$$ and when $$\frac{{n-\#_1(\sigma_n^2)} }{\sqrt{n}} \xrightarrow[n\to\infty]{\mathbb{P}} 0.$$
The first case,  by \eqref{leq4},  we obtain 
 \begin{align*} 
    \ell(\lambda(\tau_n)) \le \ell(\lambda(\sigma_n))\le \ell(\lambda(\tau_n)) +1
\end{align*} 
One need then to prove the convergence only for $\tau_n$.
Fix $\varepsilon>0,$ by conditioning on $fix(\sigma_n)=S$, $\tau_n$ is a conjugacy invariant permutations  of size $n-card(s)$ satisfying $\frac{\#(\tau_n)} {n-card(S)^\frac{1}{6}} <\varepsilon$ with high probability  and one can has by Lemma~\ref{l16}.

In the second case, we  use \citep[Theorem 30]{KAMMOUN202276} by choosing   $V_n$,
 $I_n$, $V_n^i $, $Class(\sigma)$ and 
 $i^*_n$ as in the proof of Theorem~\ref{thm2}. In this case we only need the convergence for the reference mesure i.e. we reduce the problem the proving the result for random conjugacy invaraint involution. 
 Now Let $\sigma_n$ be a random conjugacy invaraint permutation. Conditionally on $fix(\sigma_n)=S$, 
 $\tau_n$ is fixed point free convolution and one can conclude since 
  \begin{align*} 
    \ell(\lambda(\tau_n)) \le \ell(\lambda(\sigma_n))\le \ell(\lambda(\tau_n)) +1.
\end{align*}

  \end{itemize}

\end{proof}

\begin{proof}[Proof of Lemma~\ref{main_lemma}]
The equation \eqref{leq3} with $m=2$ gives that 
$$\lambda_2(\sigma_n) \leq  \lambda_2(\tau_n)+ \lambda_1(\tau_n) \leq  2 \lambda_1(\tau_n) $$

By Lemma~\ref{lem11}, with high probability $\tau_n$ is a conjugacy invariant random permutation without fixed points with a total number of cycles less than $\sqrt{n}$. We obtain $$\frac{\tau_n}{\sqrt{n}} \to 2 $$ 

\end{proof}
\begin{appendix}
\section{RS image for permutations with a microscopic number of fixed points}
\begin{proposition}\citep{MR1845180} \label{BRProp}
    Assume that $(\sigma_n)$ is a uniform involution of size $n$. 
    \[\lim_{n\to \infty} \mathbb{P}\left(\frac{\ell(\sigma_n)-2\sqrt{n}}{n^\frac 16}\leq s\right)=\lim_{n\to \infty} \mathbb{P}\left(\frac{\underline{\ell}(\sigma_n)-2\sqrt{n}}{n^\frac 16}\leq s\right)=F_1(s).
\]
    Assume that $(\sigma_n)$ is a uniform  fixed point free involution of size $n$.
  \[  \lim_{n\to \infty} \mathbb{P}\left(\frac{\ell(\sigma_n)-2\sqrt{n}}{n^\frac 16}\leq s\right)=\lim_{n\to \infty} \mathbb{P}\left(\frac{\underline{\ell}(\sigma_n)-2\sqrt{n}}{n^\frac 16}\leq s\right)=F_4(s).
\]
\end{proposition}

\begin{proposition}
    \label{Meliot} \citep{Mliot2011KerovsCL} 
Assume that $(\sigma_n)$ is a uniform involution of size $n$. Then 
for all $\varepsilon>0$,
\begin{align*}
\lim_{n\to \infty} \mathbb{P}\left(\sup_{s\in \mathbb{R}} \left|\frac{1}{\sqrt{2n}}L_{\lambda(\sigma_n)}\left({s}{\sqrt{2n}}\right)-\Omega(s)\right|<\varepsilon\right) =1.
\end{align*}
\end{proposition}

\begin{corollary}
    \label{Meliot_kam}
Assume that $(\sigma_n)$ is a uniform fixed point free involution of size $n$. Then 
for all $\varepsilon>0$,
\begin{align*}
\lim_{n\to \infty} \mathbb{P}\left(\sup_{s\in \mathbb{R}} \left|\frac{1}{\sqrt{2n}}L_{\lambda(\sigma_n)}\left({s}{\sqrt{2n}}\right)-\Omega(s)\right|<\varepsilon\right) =1.
\end{align*}    
\end{corollary}

\begin{lemma}\label{VCthm} \cite[Theorem 8]{2018arXiv180505253S} 
Assume that the sequence of random permutations  $(\sigma_n)_{n\geq 1}$ is stable under conjugation and 
for all $\varepsilon>0$,
\begin{equation*}
\lim_{n\to \infty}\mathbb{P}\left(\frac{\#(\sigma_n)}{{n} }>\varepsilon\right) =0.
\end{equation*}
Then for all $\varepsilon>0$,
\begin{align*}
\lim_{n\to \infty} \mathbb{P}\left(\sup_{s\in \mathbb{R}} \left|\frac{1}{\sqrt{2n}}L_{\lambda(\sigma_n)}\left({s}{\sqrt{2n}}\right)-\Omega(s)\right|<\varepsilon\right) =1.
\end{align*}

\end{lemma}

\begin{lemma}\citep{2018arXiv180505253S}\label{l16}
Assume that  for all positive integer $n$, $\sigma_n$ is a conjugacy invaraint permutation of size $n$ for all $\varepsilon>0$,
\begin{equation*}
\lim_{n\to \infty}\mathbb{P}\left(\frac{\#(\sigma_n)}{n^\frac 12 }>\varepsilon\right) =0.
\end{equation*}
Then ,
\begin{equation*}
\frac{\ell(\sigma_n)}{\sqrt{2}} \xrightarrow[n\to\infty]{\mathbb{P}} 2
\end{equation*}

Morover assume that for all $\varepsilon>0$,
\begin{equation*}
\lim_{n\to \infty}\mathbb{P}\left(\frac{\#(\sigma_n)}{n^\frac 16 }>\varepsilon\right) =0.
\end{equation*}
Then  for all  $s \in \mathbb{R}$,
\begin{equation*}
\lim_{n\to \infty} \mathbb{P}\left(\frac{\ell(\sigma_n)-2\sqrt{n}}{n^\frac 16}\leq s\right)=\lim_{n\to \infty} \mathbb{P}\left(\frac{\underline{\ell}(\sigma_n)-2\sqrt{n}}{n^\frac 16}\leq s\right)=F_2(s).
\end{equation*}
\end{lemma}

\end{appendix}

\end{document}